\documentclass[10pt,reqno]{amsart}
\usepackage[applemac]{inputenc}
\usepackage[english]{babel}
\usepackage{caption}
\usepackage{amsmath}
\usepackage{bm}
\usepackage{bbm}
\usepackage{xcolor}
\usepackage{graphicx}
\usepackage{amssymb}

\usepackage{color}
\usepackage{color}\newcommand{\blue}{\color{blue}}

 \newtheorem{thm}{Theorem}[section]
 
 \newtheorem{lem}[thm]{Lemma}
 
 \newtheorem{rem}{Remark}
  \newtheorem*{rem*}{Remark}
 \numberwithin{equation}{section}
 
\newcommand\eps{\varepsilon}
\newcommand\vb{|}
\newcommand\norm{\|}

 \newcommand{\be}[1]{\begin{equation}\label{#1}}
\newcommand{\ee}{\end{equation}}

\def\squarebox#1{\hbox to #1{\hfill\vbox to #1{\vfill}}}

\begin{document}

\title[ Ground states  for a relativistic model in the nonrelativistic limit]{Ground states for a stationary mean-field model for a nucleon}

\author[M.J. Esteban]{Maria J. Esteban$^1$}
\address{$^1$Ceremade, Universit\'e Paris-Dauphine, Place de Lattre de Tassigny, F-75775 Paris C\'edex 16, France}
\email{esteban@ceremade.dauphine.fr}

\author[S. Rota Nodari]{Simona Rota Nodari$^{2,3}$}
\address{$^2$ CNRS, UMR 7598, Laboratoire Jacques-Louis Lions, F-75005, Paris, France}
\address{
$^3$ UPMC Univ Paris 06, UMR 7598, Laboratoire Jacques-Louis Lions, F-75005, Paris, France}
\email{rotanodari@ann.jussieu.fr}


\date{\today}

\begin{abstract}
In this paper we consider a variational problem related to a model for a nucleon interacting with the $\omega$ and $\sigma$ mesons in the atomic nucleus. The model is relativistic,  and we study it in a nuclear physics nonrelativistic limit, which is of a very different nature  than the nonrelativistic limit in the atomic physics. Ground states are shown to exist for a large class of values for the parameters of the problem, which are determined by the values of some physical constants.   
\end{abstract}

\maketitle

\section{Introduction}

This article is concerned with the existence of minimizers for the energy functional
\begin{equation}\label{eqenergyfunc}
{  \mathcal{E}(\varphi)=\int_{\mathbb{R}^3}{\frac{|\bm\sigma\cdot\nabla\varphi|^2}{(1-|\varphi|^2)_+}\,dx}-\frac{a}{2}\int_{\mathbb{R}^3}{|\varphi|^4\,dx}}
\end{equation}
under the $L^2$-normalization constraint 
\begin{equation}\label{eqconstraintnrl3d}
\int_{\mathbb R^3}{|\varphi|^2\,dx}=1.
\end{equation} 
More precisely, for a large class of values for the parameter $a$, we show the existence of solutions of the following minimization problem
\begin{equation}\label{eqdefinf}
I=\inf\left\{\mathcal{E}(\varphi)\,;\;\varphi\in X,\int_{\mathbb R^3}{|\varphi|^2\,dx}=1\right\}\,,
\end{equation}
where
\begin{equation}\label{eqfunctspace}
X=\left\{\varphi\in   L^2(\mathbb{R}^3,\mathbb C^2)\;;\;\int_{\mathbb{R}^3}{\frac{|\bm\sigma\cdot\nabla\varphi|^2}{(1-|\varphi|^2)_+}\,dx}<+\infty\right\}\, .
\end{equation}
We remind that $\bm\sigma$ denotes the vector of Pauli matrices $(\sigma_1,\sigma_2,\sigma_3)$,
$$\sigma _1=\left( \begin{matrix} 0 & 1
\\ 1 & 0 \\ \end{matrix} \right),\quad  \sigma_2=\left( \begin{matrix} 0 & -i \\
i & 0 \\  \end{matrix}\right),\quad  \sigma_3=\left( 
\begin{matrix} 1 & 0\\  0 &-1\\  \end{matrix}\right) \, .$$

The Euler-Lagrange equation of the energy functional $\mathcal E$ under the $L^2$-normalization constraint is given by the second order equation
\begin{equation}\label{eqEL}
-\bm\sigma\cdot\nabla\left(\frac{\bm\sigma\cdot\nabla\varphi}{{ (1-|\varphi|^2)_+}}\right)+\frac{|\bm\sigma\cdot\nabla\varphi|^2}{{ (1-|\varphi|^2)_+^2}}\varphi-a|\varphi|^2\varphi+b\varphi=0\,,
\end{equation}
where $b$ is the Lagrange multiplier associated with the $L^2$-constraint \eqref{eqconstraintnrl3d}.
Hence a solution of the minimization problem \eqref{eqdefinf} is a solution of the equation \eqref{eqEL}. Moreover, Lemma \ref{lempropX} below proves that  any $\varphi\in X$ satisfies $|\varphi|^2\leq 1$ a.e. in $\mathbb R^3$. So,  a minimizer for \eqref{eqdefinf} is actually a solution of 
\begin{equation}\label{eqNRLD3d}
-\bm\sigma\cdot\nabla\left(\frac{\bm\sigma\cdot\nabla\varphi}{1-|\varphi|^2}\right)+\frac{|\bm\sigma\cdot\nabla\varphi|^2}{(1-|\varphi|^2)^2}\varphi-a|\varphi|^2\varphi+b\varphi=0.
\end{equation}
Solutions of \eqref{eqNRLD3d} which are minimizers for $I$ are called ground states.

The equation \eqref{eqNRLD3d}  is a equivalent to the system
\begin{equation}\label{eqdiracnrl}
\left\{
\begin{aligned}
&i\bm{\sigma}\cdot\nabla\chi+|\chi|^2\varphi-a|\varphi|^2\varphi+b\varphi=0,\\
&-i\bm{\sigma}\cdot\nabla\varphi+\left(1-|\varphi|^2\right)\chi=0\,.
\end{aligned}
\right.
\end{equation}
As we formally derived in a previous paper (\cite{estebanrotanodarirad}), this system is the nuclear physics nonrelativistic limit of the $\sigma$-$\omega$ relativistic mean-field model (\cite{waleckasigmaomega, walecka}) in the case of a single nucleon. 

In \cite{estebanrotanodarirad}, we proved the existence of square integrable solutions of (\ref{eqdiracnrl}) in the particular form
\begin{equation}\label{eqsolrad}
\left(\begin{aligned}\varphi(x)\\\chi(x)\end{aligned}\right)=\left(
\begin{aligned}
&g(r)\left(\begin{aligned}1\\0\end{aligned}\right)\\
&i f(r)\left(\begin{aligned}&\cos\vartheta\\&\sin\vartheta e^{i\phi}\end{aligned}\right)
\end{aligned}
\right)\,,
\end{equation}
where $f$ and $g$ are real valued radial functions.  This ansatz corresponds to  particles with minimal angular momentum, that is, $j=1/2$ (for instance, see \cite{Thaller}). In this model, the equations for $f$ and $g$ read as follows:
\begin{equation}\label{eqrad}
\left\{\begin{aligned}
f'+\frac{2}{r}f&=g(f^2-a g^2+b)\,,\\
g'&=f(1-g^2)\,,
\end{aligned}\right.
\end{equation}
where we assumed $f(0)=0$ in order to avoid solutions with singularities at the origin, and we showed that given $a,b>0$ such that $a-2b>0$, there exists at least one nontrivial solution of (\ref{eqrad}) such that 
\begin{equation}\label{limiting}(f(r), g(r)) \longrightarrow (0,0)\quad\mbox{as}\quad  r\longrightarrow +\infty
\,.
\end{equation}
 
In this paper, we prove the existence of  solutions of the above nuclear physics nonrelativistic limit of the $\sigma$-$\omega$ relativistic mean-field model without considering any particular ansatz for the nucleon's wave function.

Note that \eqref{eqNRLD3d} is the Euler-Lagrange equation of the energy functional 
\begin{equation}\label{eqfunctF}
{  \mathcal{F}(\varphi)=\int_{\mathbb{R}^3}{\frac{|\bm\sigma\cdot\nabla\varphi|^2}{1-|\varphi|^2}\,dx}-\frac{a}{2}\int_{\mathbb{R}^3}{|\varphi|^4\,dx}}
\end{equation}
under the $L^2$ normalization constraint. In the Appendix, we prove that the energy functional $\mathcal F$ is not bounded from below. So,  trying to find solutions of \eqref{eqNRLD3d} which minimize the energy ${\mathcal F}$ is hopeless and the definition of ground states for \eqref{eqNRLD3d} based on this functional is not clear.

In our previous work (\cite{estebanrotanodarirad}), we showed that for all the solutions of \eqref{eqrad}  which are square integrable,  $g^2(r)<1$ in $[0,+\infty)$. Hence, according to this result, we conjecture that a solution of \eqref{eqNRLD3d} has to satisfy $\vb \varphi\vb^2\le 1$ a.e. in $\mathbb R^3$. As we prove in the Appendix, this assumption is also justified  when we consider the intermediate model 
\begin{equation}\label{eqphireal}\varphi=\left(\begin{array}{c}u\\ 0\end{array}\right)\end{equation} with $u:\mathbb{R}^3\rightarrow \mathbb{R}$ and $a>b$. Moreover, in the physical literature  finite nuclei are described via functions $\varphi$ such that, in the right units, $|\varphi|^2\le 1$ and $|\varphi|$ is rather flat near the center of the nucleus, and is equal to $0$ outside it, see \cite{ring, greinermaruhn}.

Note that if $\vb \varphi\vb^2\le 1$ a.e. in $\mathbb R^3$, then $\mathcal F(\varphi)=\mathcal E(\varphi)$, and the ground states of \eqref{eqNRLD3d} can be defined without further specification as the minimizers of $\mathcal E$.

The main result of our paper is the following
\begin{thm}\label{thmcc} If $I<0$ there exists a minimizer of (\ref{eqdefinf}). Moreover, $I<0$ if and only if $a>a_0$ where $a_0$ is a strictly positive constant. In particular, $10.96\approx\frac{2}{S^2}<a_0<48.06$, where $S$ the best constant in the Sobolev embedding of $H^1(\mathbb R^3)$ into $L^6(\mathbb R^3)$.
\end{thm}

\begin{rem} The upper estimate for $a_0$ is obtained by using a particular test function and is probably not optimal .
\end{rem}

The proof of the above theorem is an application of the concentration-compactness principle (\cite{lionscc1, lionscc2}) with some new ingredients. The main new difficulty is due to the presence of the term $\int_{\mathbb{R}^3}{\frac{|\bm\sigma\cdot\nabla\varphi|^2}{(1-|\varphi|^2)_+}\,dx}$ in the energy functional. As we will see below, to rule out the dichotomy case in the concentration-compactness lemma we have to choose \emph{ad-hoc} cut-off functions allowing us to deal with possible singularities of the integrand. This is also necessary in order to show the localization properties of $\int_{\mathbb{R}^3}{\frac{|\bm\sigma\cdot\nabla\varphi|^2}{(1-|\varphi|^2)_+}\,dx}$.


In the next section, we will establish a concentration-compactness lemma in $X$ and then apply it to prove our main result.  The Appendix contains some auxiliary results about various properties of the model problem that we consider here.

\section{Proof of Theorem \ref{thmcc}}\label{secproofthmcc}
To prove this theorem, we are going to apply a concentration-compactness lemma that we state below. The reader may refer to \cite{lionscc1} and \cite{lionscc2} for more details on this kind of approach. {  The particular shape of the energy functional, where the kinetic energy term is multiplied by a function which could present singularities as $|\varphi|$ gets close to $1$ creates some complications in the use of concentration-compactness, that we deal with by using very particular cut-off functions.}

Let us introduce
\begin{equation}\label{eqdefinfnu}
I_\nu=\inf\left\{\mathcal{E}(\varphi)\;;\;\varphi\in X,\int_{\mathbb R^3}{|\varphi|^2\,dx}=\nu\right\}
\end{equation}
where $\nu>0$ and $I_1=I$, and we make a few preliminary observations. 

\begin{lem}[\cite{serebound}]\label{lempropX}
Let $\varphi\in X$. Then, $\varphi\in H^1(\mathbb R^3,\mathbb C^2)$ and  $|\varphi|^2\le 1$ a.e. in $\mathbb R^3$.
\end{lem}

\begin{proof}
First, by a straightforward calculation, we obtain
\begin{align*}
\int_{\mathbb{R}^3}{{|\nabla\varphi|^2}\,dx}=\int_{\mathbb{R}^3}{{|\bm\sigma\cdot\nabla\varphi|^2}\,dx}\le\int_{\mathbb{R}^3}{\frac{|\bm\sigma\cdot\nabla\varphi|^2}{(1-|\varphi|^2)_+}\,dx}<+\infty.
\end{align*}
Hence, $\varphi \in H^1(\mathbb R^3,\mathbb C^2)$. Next, let $n\in \mathbb C^2$ such that $\vb n\vb=1$.
Note that for $\varphi\in X$,  $\mathbbm 1_{\mathrm{Re}\{n\cdot\varphi\}\ge 1}(\bm\sigma\cdot\nabla\varphi)=0$, a.e. in $\mathbb{R}^3$.    Define the functions $f=(\mathrm{Re}\{n\cdot\varphi\}-1)_+$ and $\psi=fn$. (Note that for $2$ complex vectors $A, B\in \mathbb C^2$,  $A\cdot B$ denotes the scalar product $\Sigma_{i=1}^2 \overline{A}_iB$, where $\overline{z}$ stands for the complex conjugate of any complex number $z$).

We have $f\in H^1(\mathbb R^3,\mathbb R)$ and $\psi \in H^1(\mathbb R^3,\mathbb C^2)$. Moreover, for $k=1,2,3$,
$$
\partial_k \psi=\partial_k f\, n\quad\mbox{ and }\quad \partial_k f=\mathrm{Re}\left\{n\cdot
\partial_k \varphi\right\}\mathbbm{1}_{\mathrm{Re}\{n\cdot\varphi\}\ge 1} {  = n\cdot \partial_k\psi}\,.
$$
Hence, we obtain 
\begin{align*}
\int_{\mathbb R^3}\vb\nabla f\vb^2\,dx=& \int_{\mathbb R^3}\vb\nabla \psi\vb^2\,dx=\int_{\mathbb R^3}\sum_{k=1}^3\mathrm{Re} \left\{\mathrm{Re}\left\{n\cdot
\partial_k \varphi\right\}n\cdot\partial_k \psi \right\}\,dx\\
=&\int_{\mathbb R^3}\sum_{k=1}^3\mathrm{Re}\left\{n\cdot
\partial_k \varphi\right\}\mathrm{Re} \left\{n\cdot\partial_k \psi \right\}\,dx=\int_{\mathbb R^3}\sum_{k=1}^3\mathrm{Re}\left\{\partial_k f\,n\cdot
\partial_k \varphi\right\}\,dx\\
=& \int_{\mathbb R^3}\mathrm{Re}\left\{\nabla \psi\cdot\nabla \varphi\right\}\,dx= \int_{\mathbb R^3}\mathrm{Re}\left\{(\bm\sigma\cdot\nabla \psi)\cdot(\bm\sigma\cdot\nabla \varphi)\right\}\,dx\\
=& \int_{\mathbb R^3}\mathrm{Re}\left\{(\bm\sigma\cdot\nabla \psi)\cdot\mathbbm 1_{\mathrm{Re}\{n\cdot\varphi\}\ge 1}(\bm\sigma\cdot\nabla \varphi)\right\}\,dx=0
\end{align*}
As a consequence, $f=0$ a.e. in $\mathbb R^3$ that means $\mathrm{Re}\{n\cdot\varphi\}\le1$ a.e. for all $n\in \mathbb C^2$ such that $\vb n\vb=1$. {  This clearly  implies that $\vb \varphi\vb \le 1$ a.e. in $\mathbb R^3$. }
\end{proof}

In what follows, we say that a sequence $\{\varphi_n\}_n$ is $X$-bounded if there exists a positive constant $C$ independent of $n$ such that
\begin{equation}\label{eqXbounded}
\norm\varphi_n\norm^2_{L^2}+\int_{\mathbb{R}^3}{\frac{|\bm\sigma\cdot\nabla\varphi_n|^2}{(1-|\varphi_n|^2)_+}\,dx}\le C\,.
\end{equation}

\begin{lem}\label{borne-sm} Let $\{\varphi_n\}_n$ be a minimizing sequence of (\ref{eqdefinfnu}), then $\{\varphi_n\}_n$ is $X$-bounded, bounded in $H^1(\mathbb R^3)$ and $I_{\nu}>-\infty$.
\end{lem}
\begin{proof} Indeed, since $\{\varphi_n\}_n$ is a minimizing sequence, there exists a constant $C$ such that
\begin{align*}
C\ge \mathcal E (\varphi_n)&\ge\int_{\mathbb{R}^3}{\frac{|\bm\sigma\cdot\nabla\varphi_n|^2}{(1-|\varphi_n|^2)_+}\,dx}-\frac{a}{2}\nu\ge \int_{\mathbb{R}^3}{{|\bm\sigma\cdot\nabla\varphi_n|^2}\,dx}-\frac{a}{2}\nu\\
&= \int_{\mathbb{R}^3}{{|\nabla\varphi_n|^2}\,dx}-\frac{a}{2}\nu\ge -\frac{a}{2}\nu.
\end{align*}
As a conclusion, $\|\varphi_n\|_{H^1}$ is bounded independently of $n$ and $I_\nu$ is bounded from below.
\end{proof}

\begin{lem}\label{inegalites-strictes} For all $\nu\in(0,1)$,  $I_\nu\le0$. Moreover, 
the strict inequality $I<0$ is equivalent to the strict concentration-compactness inequalities 
\begin{equation}\label{eqstrictcc}
I<I_\nu+I_{1-\nu}\quad,\quad\forall\nu\in(0,1)\,.
\end{equation}
\end{lem}
\begin{proof}
Indeed, let $\varphi\in\mathcal D(\mathbb R^3)$ such that  $\int_{\mathbb R^3}{|\varphi|^2}=\nu$ and $\int_{\mathbb{R}^3}{\frac{|\bm\sigma\cdot\nabla\varphi|^2}{(1-|\varphi|^2)_+}\,dx}<+\infty$, and let $\varphi_\gamma(x)=\gamma^{-3/2}\varphi(\gamma^{-1}x)$ for $\gamma>1$. Then
\begin{align*}
I_\nu\le \mathcal{E}(\varphi_\gamma)=\frac{1}{\gamma^2}\int_{\mathbb{R}^3}{\frac{|\bm\sigma\cdot\nabla\varphi|^2}{\left(1-\frac{1}{\gamma^3}|\varphi|^2\right)_+}\,dx}-\frac{1}{\gamma^3}\frac{a}{2}\int_{\mathbb{R}^3}{|\varphi|^4\,dx}\,,
\end{align*}
and letting $\gamma\to+\infty$, we prove $I_\nu\le 0$.

By a scaling argument, we obtain 
$$
I_{\vartheta\nu}\le\inf\left\{\vartheta^{1/3}\int_{\mathbb{R}^3}{\frac{|\bm\sigma\cdot\nabla\varphi|^2}{(1-|\varphi|^2)_+}\,dx}-\frac{\vartheta\,a}{2}\int_{\mathbb{R}^3}{|\varphi|^4\,dx}|\varphi\in X, \int_{\mathbb R^3}{|\varphi|^2\,dx}=\nu\right\},
$$
and, if $I_\nu<0$, we may restrict the infimum $I_\nu$ to elements $\varphi$ satisfying
$$
K(\varphi)=\int_{\mathbb{R}^3}{\frac{|\bm\sigma\cdot\nabla\varphi|^2}{(1-|\varphi|^2)_+}\,dx}\ge \delta>0\,,
$$
{  for some $\delta>0$.} Indeed, if there is a minimizing sequence $\{\varphi_n\}_n$ of $I_\nu$ such that $K(\varphi_n)\xrightarrow[n]{}0$, then, by Sobolev embeddings, $\varphi_n\xrightarrow[n]{}0$ in $L^p(\mathbb R^3)$ for $2<p\le 6$ and $I_\nu\ge 0$. As a conclusion, if $I_\nu<0$, then, for all $\vartheta>1$ and for all $\nu>0$,
\begin{equation}\label{eqinequalitycc}
I_{\vartheta\nu}<\vartheta\inf\left\{\mathcal{E}(\varphi)|\varphi\in X, K(\varphi)>0, \int_{\mathbb R^3}{|\varphi|^2\,dx}=\nu\right\}=\vartheta I_\nu\,.
\end{equation}
 Hence, a straightforward argument (see lemma II.$1$ of  \cite{lionscc1}) proves that (\ref{eqstrictcc}) is equivalent to $I<0$.
\end{proof}

{  In order to prove Theorem \ref{thmcc} we need to analyse the possible behaviour of minimizing sequences for $I$. This is done in the following lemma.}

\begin{lem}\label{lemcc}
Let $\{\varphi_n\}_n$ be a $X$-bounded sequence such that  $\int_{\mathbb R^3}{|\varphi_n|^2\,dx}=1$ for all $n\ge 0$. Then there exists a subsequence that we still denote by $\{\varphi_n\}_n$ such that one of the following properties holds:
\begin{enumerate}
\item\label{itemcccompactness} Compactness up to a translation: there exists a sequence $\{y_n\}_n \subset \mathbb{R}^3$ such that, for every $\varepsilon>0$, there exists $ 0<R<\infty$ with
$$
\int_{B\left(y_n,R\right)}{|\varphi_n|^2\,dx}\ge 1-\varepsilon;
$$
\item\label{itemccvanishing} Vanishing: for all $0<R<\infty$
$$
\sup\limits_{y\in \mathbb{R}^3}\int_{B\left(y,R\right)}{|\varphi_n|^2\,dx}\xrightarrow[n]{} 0;
$$
\item\label{itemccdichotomy} Dichotomy: there exist $\alpha \in(0,1)$ and $n_0\ge 0$ such that
there exist two $X$-bounded sequences, $\{\varphi_1^n\}_{n\ge n_0}$ and $\{\varphi_2^n\}_{n\ge n_0}$, 
satisfying the following properties:
\begin{align}\label{eqdichconvlp}
\|\varphi_n-(\varphi_1^n+\varphi_2^n)\|_{L^p}\xrightarrow[n]{} 0,&\mbox{ for } 2\le p<6,
\end{align}
and
\begin{equation}\label{eqdichmass}
\begin{aligned}
\int_{\mathbb R^3}{|\varphi_1^n|^2\,dx}\xrightarrow[n]{}\alpha&\mbox{ and}& \int_{\mathbb R^3}{|\varphi_2^n|^2\,dx}\xrightarrow[n]{}1-\alpha,
\end{aligned}
\end{equation}
\begin{equation}\label{eqdichsupp}
\mathrm{dist}(\mathrm{supp}\,\varphi_1^n,\mathrm{supp}\,\varphi_2^n)\xrightarrow[n]{}+\infty.
\end{equation}
Moreover, in this case we have that 
\begin{equation}\label{eqlocalizfunct}
\liminf_{n\to+\infty}\mathcal{E}(\varphi_n)-\mathcal{E}(\varphi_1^n)-\mathcal{E}(\varphi_2^n)\ge 0\,,
\end{equation}
{  which implies $I\ge I_\alpha+I_{1-\alpha}$.}
\end{enumerate}
\end{lem}

\begin{proof}[Proof of Lemma \ref{lemcc}] Let $\{\varphi_n\}_n$ be a $X$-bounded sequence such that  $\int_{\mathbb R^3}{|\varphi_n|^2\,dx}=\nu$ for all $n\ge 0$. We remind that $X$-bounded means that there exists $C>0$ such that 
$$
\norm\varphi_n\norm^2_{L^2}+\int_{\mathbb{R}^3}{\frac{|\bm\sigma\cdot\nabla\varphi_n|^2}{(1-|\varphi_n|^2)_+}\,dx}\le C.
$$
Moreover, thanks to Lemma \ref{lempropX}, if $\{\varphi_n\}_n$ is a $X$-bounded sequence then $\{\varphi_n\}_n$ {  is bounded in $L^\infty$ (by the constant $1$)} and in $H^1({\mathbb R^3})$.
Then, along the lines of \cite{lionscc1}, we introduce the so-called Lévy concentration functions
\begin{align}
\label{eqlevyQ}&Q_n(R)=\sup\limits_{y\in \mathbb R^3}\int_{\vb x-y\vb<R}|\varphi_n|^2\,dx,\\
\label{eqlevyK}&K_n(R)=\sup\limits_{y\in \mathbb R^3}\int_{\vb x-y\vb<R}{\frac{|\bm\sigma\cdot\nabla\varphi_n|^2}{(1-|\varphi_n|^2)_+}\,dx}
\end{align}
for $R>0$. Note that $Q_n$ and $K_n$ are continuous non-decreasing functions on $[0,+\infty)$, such that for all $n\ge 0$ and for all $R>0$
$$
Q_n(R)+K_n(R)\le C
$$
since $\{\varphi_n\}_n$ is $X$-bounded. Then, up to a subsequence, we have for all $R>0$
\begin{align}
\label{eqlevyQconv}&Q_n(R)\xrightarrow[n]{} Q(R),\\
\label{eqlevyKconv}&K_n(R)\xrightarrow[n]{} K(R),
\end{align} 
where $Q$ and $K$ are nonnegative, non-decreasing functions. Clearly, we have that 
$$
\alpha=\lim_{R\to+\infty}Q(R) \in [0,1],
$$ 
and we denote $l=\lim_{R\to+\infty}K(R)$.

If $\alpha=0$, then the situation (\ref{itemccvanishing}) of the lemma arises as a direct consequence of Definition (\ref{eqlevyQ}). If $\alpha=1$, then \eqref{itemcccompactness} follows, see \cite{lionscc1} for details.
Assume that $\alpha\in(0,1)$, we have to show that (\ref{itemccdichotomy}) holds.

{  First of all, consider $\varepsilon>0$, small, and $ R_\varepsilon>0$ such that $Q(R_\varepsilon)=\alpha-\varepsilon$ and $K(R_\varepsilon)\le l-\varepsilon$. Then, for $n$ large enough,
 $$ Q_n(R_\varepsilon)-Q(R_\varepsilon)<1/n, \quad K_n(R_\varepsilon)-K(R_\varepsilon)<1/n\,,$$
and by definition of the L\'evy functions $Q_n$, extracting subsequences if necessary, there exists 
$y_n\in  \mathbb R^3$ such that
\begin{align*}
&\left\vb \int_{\vb x-y_n\vb<R_\varepsilon}|\varphi_n|^2\,dx-Q_n(R_\varepsilon)\right\vb\le \frac{1}{n},\\
&\left\vb\int_{\vb x-y_n\vb<R_\varepsilon}{\frac{|\bm\sigma\cdot\nabla\varphi_n|^2}{(1-|\varphi_n|^2)_+}\,dx-K_n(R_\varepsilon)}\right\vb\le\frac{1}{n}\,.
\end{align*}
Next define $R_n>R_\varepsilon$ such that 
$$
 \int_{R_\varepsilon<\vb x-y_n\vb<R_n}|\varphi_n|^2\,dx= \frac{3}{n}+\varepsilon\,.
$$
Necessarily, $R_n\to +\infty$ as $n\to +\infty$. Indeed, if $R_n\le M$ for some $M>0$, then $Q(M)>\alpha$, which is impossible. We then deduce that for $n$ large enough,
$$
\label{eqconverrQ}\int_{\frac{R_n}8\le\vb x-y_n\vb\le R_n}|\varphi_n|^2\,dx\le \frac{3}{n}+\varepsilon\,
$$}
Let $\xi$, $\zeta$ be cut-off functions: $\xi, \zeta\in \mathcal D (\mathbb R^3)$ such that 
\begin{align*}
&\xi(x)=\left\{\begin{aligned}&1&&|x|\le1\\&1-\exp\left(1{-\dfrac{1}{1-\exp\left({1-\frac{1}{2-\vb x\vb}}\right)}}\right)&&1< \vb x \vb<2\\&0&&|x|\ge2\end{aligned}\right.\\&\zeta(x)=\left\{\begin{aligned}&0&&|x|\le1\\&\exp\left({1-\dfrac{1}{1-\exp\left({1-\frac{1}{2-\vb x\vb}}\right)}}\right)&&1< \vb x \vb<2\\&1&&|x|\ge2\end{aligned}\right.,
\end{align*}
and let $\xi_\mu$, $\zeta_\mu$ denote $\xi\left(\frac{\cdot}{\mu}\right)$, $\zeta\left(\frac{\cdot}{\mu}\right)$. We define
\begin{align}\label{eqphi1dich}
\varphi_1^n(\cdot)=\xi_{\frac{R_n}{8}}(\cdot-y_n)\varphi_n(\cdot)=\xi_{\frac{R_n}{8},y_n}(\cdot)\varphi_n(\cdot)\\
\label{eqphi2dich}
\varphi_2^n(\cdot)=\zeta_{\frac{R_n}{2}}(\cdot-y_n)\varphi_n(\cdot)=\zeta_{\frac{R_n}{2},y_n}(\cdot)\varphi_n(\cdot)
\end{align} 
with $R_n\to+\infty$. (\ref{eqdichsupp}) follows easily from these definitions. Furthermore, \eqref{eqdichconvlp} and (\ref{eqdichmass}) are obtained in the following way:  
\begin{align*}
\lim\limits_{n\to+\infty}\int_{\mathbb R^3}\vb\varphi_n-(\varphi_1^n+\varphi_2^n)\vb^2\,dx&=\lim\limits_{n\to+\infty}\int_{\frac{R_n}{8}\le\vb x-y_n\vb\le R_n}\vb(1-\xi_{\frac{R_n}{8}}-\zeta_{\frac{R_n}{2}})\varphi_n\vb^2\,dx\\
&\le
\lim\limits_{n\to+\infty}\int_{\frac{R_n}{8}\le\vb x-y_n\vb\le R_n}\vb\varphi_n\vb^2\,dx \le \eps,
\end{align*}
 Now by taking a sequence of $\eps$ tending to $0$, and by taking a diagonal sequence of the functions $\varphi_n$, and calling it by the same name, we find
 $$\int_{\frac{R_n}{8}\le\vb x-y_n\vb\le R_n}\vb\varphi_n\vb^2\,dx \xrightarrow[n]{} 0,$$
and, since $\{\varphi_1^n\}_n$ and  $\{\varphi_2^n\}_n$ are bounded in $H^1(\mathbb{R}^3)$, we also obtain
\begin{equation*}\label{eqconvdico}
\lim\limits_{n\to+\infty}\|\varphi_n-(\varphi_1^n+\varphi_2^n)\|_{L^p}\xrightarrow[n]{} 0,
\end{equation*}
for $2\le p< 6$.
Next, we have to prove that  $\{\varphi_1^n\}_{n\ge n_0}$ and  $\{\varphi_2^n\}_{n\ge n_0}$ are $X$-bounded. To this purpose, we show that 
\begin{equation}\label{eqphi1energylimit}
\lim\limits_{n\to+\infty}\int_{\mathbb R^3}{\frac{|\bm\sigma\cdot\nabla\varphi_1^n|^2}{(1-|\varphi_1^n|^2)_+}\,dx}-\int_{\mathbb R^3}{\frac{\xi^2_{\frac{R_n}{8},y_n}|\bm\sigma\cdot\nabla\varphi_n|^2}{(1-|\varphi_1^n|^2)_+}\,dx}=0\,
\end{equation}
and 
\begin{equation}\label{eqphi2energylimit}
\lim\limits_{n\to+\infty}\int_{\mathbb R^3}{\frac{|\bm\sigma\cdot\nabla\varphi_2^n|^2}{(1-|\varphi_2^n|^2)_+}\,dx}-\int_{\mathbb R^3}{\frac{\zeta^2_{\frac{R_n}{2},y_n}|\bm\sigma\cdot\nabla\varphi_n|^2}{(1-|\varphi_2^n|^2)_+}\,dx}=0\,
\end{equation}

Indeed, if (\ref{eqphi1energylimit}) and \eqref{eqphi2energylimit} hold, we obtain that for all $\eps>0$, there exists $n_0\ge 0$ such that for all $n\ge n_0$,  we have
{  \begin{align*}
\int_{\mathbb R^3}{\frac{|\bm\sigma\cdot\nabla\varphi^n_1|^2}{(1-|\varphi^n_1|^2)_+}\,dx}&\le\int_{\mathbb R^3}{\frac{\xi^2_{\frac{R_n}{8},y_n}|\bm\sigma\cdot\nabla\varphi_n|^2}{(1-|\varphi^n_1|^2)_+}\,dx}+o(1)_{n\to +\infty}\\
&\qquad \le \int_{\mathbb R^3}{\frac{|\bm\sigma\cdot\nabla\varphi_n|^2}{(1-|\varphi_n|^2)_+}\,dx}+o(1)_{n\to +\infty}
\le C+o(1)_{n\to +\infty}, 
\end{align*}
and
\begin{align*}
\int_{\mathbb R^3}{\frac{|\bm\sigma\cdot\nabla\varphi_2^n|^2}{(1-|\varphi^n_2|^2)_+}\,dx}&\le\int_{\mathbb R^3}{\frac{\zeta^2_{\frac{R_n}{2},y_n}|\bm\sigma\cdot\nabla\varphi_n|^2}{(1-|\varphi^n_2|^2)_+}\,dx}+o(1)_{n\to +\infty}\\ &\qquad \le \int_{\mathbb R^3}{\frac{|\bm\sigma\cdot\nabla\varphi_n|^2}{(1-|\varphi_n|^2)_+}\,dx}+o(1)_{n\to +\infty}
\le C+o(1)_{n\to +\infty}.
\end{align*}

To prove (\ref{eqphi1energylimit})  we proceed as follows. We remark that 
$$
\int_{\mathbb R^3}{\frac{|\bm\sigma\cdot\nabla\varphi_1^n|^2}{(1-|\varphi_1^n|^2)_+}\,dx}-\int_{\mathbb R^3}{\frac{\xi^2_{\frac{R_n}{8},y_n}|\bm\sigma\cdot\nabla\varphi_n|^2}{(1-|\varphi_1^n|^2)_+}\,dx}= A_n+B_n\,,$$
where
\begin{align*}
A_n:=\int_{\mathbb R^3}{\frac{|\bm\sigma\cdot(\nabla\xi_{\frac{R_n}{8},y_n})\varphi_n|^2}{(1-|\varphi_1^n|^2)_+}\,dx}
&=\int_{\frac{R_n}{8}\le\vb x-y_n\vb\le\frac{R_n}{4}}{\frac{|\bm\sigma\cdot(\nabla\xi_{\frac{R_n}{8},y_n})\varphi_n|^2}{(1-|\varphi_1^n|^2)_+}\,dx}\\
&\le \int_{\frac{R_n}{8}\le\vb x-y_n\vb\le\frac{R_n}{4}}{\frac{|\bm\sigma\cdot(\nabla\xi_{\frac{R_n}{8},y_n})\varphi_n|^2}{1-\xi_{\frac{R_n}{8},y_n}^2}\,dx}:=C_n\,,
\end{align*}
and 
$$
|B_n|\le 2 \left( C_n\right)^\frac12\left(\int_{\mathbb R^3}{\frac{|\bm\sigma\cdot\nabla\varphi_n|^2}{(1-|\varphi_n|^2)_+}\,dx}\right)^\frac12\,.
$$

Let us now prove that $C_n$ tends to $0$ as $n$ goes to $+\infty$.} Using spherical coordinates, we obtain
\begin{align*}
C_n&\le \int_{\frac{R_n}{8}}^{\frac{R_n}{4}}\int_{0}^\pi\int_0^{2\pi}{\frac{|\bm(\sigma\cdot\bm {e}_r)\, \varphi_n(s,\theta,\phi)|^2\left(\xi'_{\frac{R_n}{8}}(s)\right)^2}{1-\xi_{\frac{R_n}{8}}^2(s)}\,s^2\,\sin \theta \,ds\, d\theta\, d\phi}\\
&\le \frac{{  64}}{R_n^2} \int_{\frac{R_n}{8}}^{\frac{R_n}{4}}\int_{0}^\pi\int_0^{2\pi}{\frac{|\varphi_n(s,\theta,\phi)|^2\left(\xi'\left(\frac{8}{R_n}s\right)\right)^2}{1-\xi^2\left(\frac{8}{R_n}s\right)}\,s^2\,\sin \theta \,ds\, d\theta\, d\phi}\\
&\le \frac{{  64}}{R_n^2} \max_{1\le r \le 2}{\frac{\left(\xi'\left(r\right)\right)^2}{1-\xi^2\left(r\right)}}
\int_{0}^{+\infty}\int_{0}^\pi\int_0^{2\pi}{|\varphi_n(s,\theta,\phi)|^2\,s^2\,\sin \theta \,ds\, d\theta\, d\phi}=O\left(\frac{1}{R_n^2}\right)
\end{align*}
since $\max\limits_{1\le r \le 2}{\frac{\left(\xi'\left(r\right)\right)^2}{1-\xi^2\left(r\right)}}\le C$. Indeed, since $\xi^2\left(r\right)=1$ if and only if $r=1$, ${\frac{\left(\xi'\left(r\right)\right)^2}{1-\xi^2\left(r\right)}}$ is a continuous function on $(1,2)$. Moreover, by a straightforward calculation, we obtain $\lim\limits_{r\to 1^+}\frac{\left(\xi'\left(r\right)\right)^2}{1-\xi^2\left(r\right)}=0=\lim\limits_{r\to 2^-}\frac{\left(\xi'\left(r\right)\right)^2}{1-\xi^2\left(r\right)}$. Hence, we can conclude, that $\frac{\left(\xi'\left(r\right)\right)^2}{1-\xi^2\left(r\right)}$ is bounded in $[1,2]$.
As a conclusion, since $R_n\to+\infty$, we obtain
\begin{equation*}
\lim\limits_{n\to+\infty}\int_{\mathbb R^3}{\frac{|\bm\sigma\cdot\nabla\varphi_1^n|^2}{(1-|\varphi^n_1|^2)_+}\,dx}-\int_{\mathbb R^3}{\frac{\xi^2_{\frac{R_n}{8},y_n}|\bm\sigma\cdot\nabla\varphi_n|^2}{(1-|\varphi_1^n|^2)_+}\,dx}=0.
\end{equation*}
With the same argument, we prove (\ref{eqphi2energylimit}).

Finally, it remains to show that
$$
\liminf_{n\to+\infty}\mathcal{E}(\varphi_n)-\mathcal{E}(\varphi_1^n)-\mathcal{E}(\varphi_2^n)\ge 0.
$$
First of all, using the definitions (\ref{eqphi1dich}) and (\ref{eqphi2dich}), we obtain
$$
\lim_{n\to+\infty}\int_{\mathbb{R}^3}{\frac{|\bm\sigma\cdot\nabla\varphi_n|^2}{(1-|\varphi_n|^2)_+}\,dx}\ge\lim_{n\to+\infty}\int_{\mathbb{R}^3}{\frac{|\bm\sigma\cdot\nabla\varphi_n|^2}{(1-|\varphi_1^n|^2-|\varphi_2^n|^2)_+}\,dx}.
$$

Next, we remark that 
\begin{align*}
\int_{\mathbb{R}^3}\frac{|\bm\sigma\cdot\nabla\varphi_n|^2}{(1-|\varphi_1^n|^2-|\varphi_2^n|^2)_+}&\,dx-\int_{\mathbb R^3}{\frac{|\bm\sigma\cdot\nabla\varphi_1^n|^2}{(1-|\varphi^n_1|^2)_+}\,dx}-\int_{\mathbb R^3}{\frac{|\bm\sigma\cdot\nabla\varphi_2^n|^2}{(1-|\varphi^n_2|^2)_+}\,dx}\\
=&\int_{\mathbb{R}^3}{\frac{|\bm\sigma\cdot\nabla\varphi_n|^2}{(1-|\varphi_1^n|^2-|\varphi_2^n|^2)_+}\,dx}-\int_{\mathbb R^3}{\frac{\xi^2_{\frac{R_n}{8},y_n}|\bm\sigma\cdot\nabla\varphi_n|^2}{(1-|\varphi_1^n|^2)_+}\,dx}\\
&-\int_{\mathbb R^3}{\frac{\zeta^2_{\frac{R_n}{2},y_n}|\bm\sigma\cdot\nabla\varphi_n|^2}{(1-|\varphi_2^n|^2)_+}\,dx}+o(1)_{n\to\infty}\\
=&\int_{\mathbb{R}^3}{\frac{\left(1-\xi^2_{\frac{R_n}{8},y_n}-\zeta^2_{\frac{R_n}{2},y_n}\right)|\bm\sigma\cdot\nabla\varphi_n|^2}{(1-|\varphi_1^n|^2-|\varphi_2^n|^2)_+}\,dx}+o(1)_{n\to\infty}\\
{  \ge}  &\; {  o(1)_{n\to\infty}.}
\end{align*}

As a conclusion,
\begin{align*}
\lim_{n\to+\infty}\int_{\mathbb{R}^3}{\frac{|\bm\sigma\cdot\nabla\varphi_n|^2}{(1-|\varphi_n|^2)_+}\,dx}\ge&\lim_{n\to+\infty}\int_{\mathbb R^3}{\frac{|\bm\sigma\cdot\nabla\varphi_1^n|^2}{(1-|\varphi^n_1|^2)_+}\,dx}\\
&+\lim\limits_{n\to+\infty}\int_{\mathbb R^3}{\frac{|\bm\sigma\cdot\nabla\varphi_2^n|^2}{(1-|\varphi^n_2|^2)_+}\,dx},
\end{align*}
and, using (\ref{eqdichconvlp}) and the localization properties of $\varphi_1^n$ and $\varphi_2^n$, we have
\begin{equation*}
I=\lim_{n\to+\infty}\mathcal{E}(\varphi_n)\ge\liminf_{n\to+\infty}\mathcal{E}(\varphi_1^n)+\liminf_{n\to+\infty}\mathcal{E}(\varphi_2^n) \; {  \ge  I_\alpha + I_{1-\alpha}}\,.
\end{equation*}
\end{proof}

\noindent {\it Proof of Theorem \ref{thmcc}.} Assume that $I<0$. By Lemma \ref{borne-sm}, any minimizing sequence $\{\varphi_n\}_{n}$ is $X$-bounded, and then we can use Lemma \ref{lemcc} to it. It is easy to rule out vanishing and dichotomy whenever $I<0$.

Vanishing cannot occur. Indeed, 
If vanishing occurs, then, up to a subsequence, $\forall R<+\infty$ we have
\begin{equation}\label{eqvanishingcc}
\lim_{n\to+\infty}\sup_{y\in\mathbb R^3}\int_{B(y,R)}{|\varphi_{n}|^2}=0.
\end{equation}
This implies that $\varphi_n$ converges strongly in $L^p(\mathbb R^3)$ for $2<p<6$ and, as a consequence, $I\ge 0$. Clearly, this contradicts $I<0$.\\
Moreover, if dichotomy occurs, we have
\begin{align*}
I=\lim_{n\to+\infty}\mathcal{E}(\varphi_n)\ge\liminf_{n\to+\infty}\mathcal{E}(\varphi_1^n)+\liminf_{n\to+\infty}\mathcal{E}(\varphi_2^n)\ge I_{\alpha}+I_{1-\alpha}
\end{align*}
which contradicts Lemma \ref{inegalites-strictes}, since $I<0$.

Hence, for $n$ large enough, there exists $\{y_n\}_n\in \mathbb R^3$ such that $\forall\varepsilon>0$, $\exists R<+\infty$,
$$
\int_{B(y_n,R)}{|\varphi_n|^2}\ge 1-\varepsilon.
$$

We denote by $\tilde\varphi_n(\cdot)=\varphi_n(\cdot+y_n)$. Since $\{\tilde\varphi_n\}_n$ is bounded in $H^1$, $\{\tilde\varphi_n\}_n$ converges weakly in $H^1$, almost everywhere on $\mathbb{R}^3$ and in $L^p_{loc}$ for $2\le p<6$ to some $\tilde\varphi$. In particular, as a consequence of weak convergence in $H^1$, $\bm\sigma\cdot\nabla\tilde\varphi_n$ converges weakly to $\bm\sigma\cdot\nabla\tilde\varphi$ in $L^2$.
Moreover, thanks to the concentration-compactness argument, $\{\tilde\varphi_n\}_n$ converges strongly in $L^2$ and in $L^p$ for $2\le p<6$. 

\begin{lem}\label{lemwlsc} Let $\{f_n\}_n$ and $\{g_n\}_n$ be two sequences of functions such that $f_n:\mathbb R^3\to \mathbb R_+$, $g_n:\mathbb R^3\to\mathbb C^2$ , $f_n$ converges to $f$ a.e., $g_n$ converges weakly to $g$ in $L^2$ and there exists a constant $C$, that does not depend on $n$, such that $\int_{\mathbb R^3}f_n\vb g_n\vb^2\,dx\le C$. Then
$$
\int_{\mathbb R^3}f\vb g\vb^2\,dx\le \liminf_{n\to+\infty}\int_{\mathbb R^3}f_n\vb g_n\vb^2\,dx.
$$
\end{lem}
\begin{proof} Given a function $h:\mathbb R^3\to \mathbb R_+$, let $T_k$ be the function defined by  
$$
T_k(h)(x)=\left\{\begin{aligned}&h(x) &\mbox{if }h(x) \le k\\&k &\mbox{if } h(x) > k\end{aligned}\right.
$$
for all $k\in [0,\infty)$. Hence, the following properties are satisfied for all $k\in [0,\infty)$:
\begin{align}
\label{eqconvaeTk} &T_k(f_n)\xrightarrow[n]{}T_k(f)\quad\mbox{a.e. in } \mathbb R^3,\\
\label{eqconvL1Tkg2} &T_k(f_n)\vb g\vb^2\xrightarrow[n]{}T_k(f)\vb g\vb^2\quad \mbox{in } L^1,\\
\label{eqconvL2weakTkg} &T_k(f_n)g \underset{n}{\rightharpoonup}T_k(f)g\quad\mbox{in } L^2,\\
\label{eqconvnormL2Tkg} &\norm T_k(f_n) g\norm_{L^2}\xrightarrow[n]{}\norm T_k(f)g\norm_{L^2}, &
\end{align}
where to obtain (\ref{eqconvL1Tkg2}) and (\ref{eqconvnormL2Tkg}), we use Lebesgue's dominated convergence theorem. Moreover, as a consequence of (\ref{eqconvL2weakTkg}) and (\ref{eqconvnormL2Tkg}), we have
\begin{equation}\label{eqconvL2Tkg} 
T_k(f_n)g \xrightarrow[n]{}T_k(f)g\quad\mbox{in } L^2.
\end{equation} 
Next, we have
\begin{align*}
0\le&\liminf_{n\to+\infty}\int_{\mathbb R^3}T_k(f_n)\vb g_n-g\vb^2\,dx=\liminf_{n\to+\infty}\int_{\mathbb R^3}T_k(f_n)\vb g_n\vb^2\,dx\\
&+\liminf_{n\to+\infty}\int_{\mathbb R^3}T_k(f_n)\vb g\vb^2\,dx-\liminf_{n\to+\infty}\left(\int_{\mathbb R^3}T_k(f_n)\,{  \overline{g}_n\cdot g}\,dx+\int_{\mathbb R^3}T_k(f_n)\,{  {g}_n\cdot \overline{g}}\,dx\right)\\
=&\liminf_{n\to+\infty}\int_{\mathbb R^3}T_k(f_n)\vb g_n\vb^2\,dx+\int_{\mathbb R^3}T_k(f)\vb g\vb^2\,dx-2\int_{\mathbb R^3}T_k(f)\vb g\vb^2\,dx
\end{align*}
thanks to (\ref{eqconvL1Tkg2}), (\ref{eqconvL2Tkg}) and the fact that $g_n$ converges weakly to $g$ in $L^2$. As a consequence,
\begin{equation}\label{eqliminfk}
\int_{\mathbb R^3}T_k(f)\vb g\vb^2\,dx\le \liminf_{n\to+\infty}\int_{\mathbb R^3}T_k(f_n)\vb g_n\vb^2\,dx
\end{equation}
Since 
$$\liminf_{n\to+\infty}\int_{\mathbb R^3}T_k(f_n)\vb g_n\vb^2\,dx\le\liminf_{n\to+\infty}\int_{\mathbb R^3}f_n\vb g_n\vb^2\,dx\le C,
$$
we can pass to the limit for $k$ that goes to $+\infty$ in (\ref{eqliminfk}) and we obtain 
$$
\int_{\mathbb R^3}f\vb g\vb^2\,dx\le \liminf_{n\to+\infty}\int_{\mathbb R^3}f_n\vb g_n\vb^2\,dx.
$$
\end{proof}
By applying Lemma (\ref{lemwlsc}) to $f_n=\frac{1}{(1-\vb\tilde\varphi_n\vb^2)}_+$ and $g_n=\vb\bm\sigma\cdot\nabla\tilde\varphi_n\vb$, we obtain 
$$
\int_{\mathbb{R}^3}{\frac{|\bm\sigma\cdot\nabla\tilde\varphi|^2}{(1-|\tilde\varphi|^2)_+}\,dx}\le \liminf_{n\to+\infty}\int_{\mathbb{R}^3}{\frac{|\bm\sigma\cdot\nabla\tilde\varphi_n|^2}{(1-|\tilde\varphi_n|^2)_+}\,dx}.
$$
Hence, $\tilde\varphi\in X$, $\int_{\mathbb R^3}{|\tilde\varphi|^2\,dx}=1$, and
$$
\mathcal E (\tilde\varphi)\le \liminf_{n\to+\infty} \mathcal E (\tilde\varphi_n) \le \mathcal E (\tilde\varphi).
$$
As a conclusion, the minimum of $I$ is achieved by $\tilde\varphi$.

Finally, it remains to prove that there exists $a_0>0$ such that for all $a>a_0$ we have $I<0$. 

It is clear that $I<0$ for $a$ large enough. Since $I$ is non-increasing with respect to $a$, we may denote by $a_0$ the least positive constant such that $I<0$ for $a>a_0$. We have to prove that $a_0>0$ or in other words $I=0$ for $a$ small enough. Using Sobolev and H\"older inequalities, we find, for $\varphi\in X$ such that $\int_{\mathbb R^3}\vb \varphi\vb^2\,dx=1$,
$$
\mathcal E(\varphi)\ge \frac{1}{S^2}\left(\int_{\mathbb R^3}\vb \varphi\vb^6\,dx\right)^{1/3}-\frac{a}{2}\left(\int_{\mathbb R^3}\vb \varphi\vb^6\,dx\right)^{1/3}.
$$
Hence, if $a\le \frac{2}{S^2}$, $I=0$. This implies $a_0> \frac{2}{S^2}$. According to \cite{talenti} the best constant for the Sobolev inequality 
$$
\norm u \norm_{L^q(\mathbb R^m)}\le C \norm \nabla u \norm_{L^p(\mathbb R^m)}  
$$
with $1<p<m$ and $q=\frac{mp}{(m-p)}$ is given by 
$$
C=\pi^{-1/2}m^{-1/p}\left(\frac{p-1}{m-p}\right)^{1-1/p}\left(\frac{\Gamma(1+m/2)\Gamma(m)}{\Gamma(m/p)\Gamma(1+m-m/p)}\right)^{1/m}.
$$
In particular,
$$
S=\frac{1}{\sqrt{3\pi}}\left(\frac{4}{\sqrt{\pi}}\right)^{1/3},
$$
and
$$
\frac{2}{S^2}=\frac{3\pi^{4/3}}{2^{1/3}}\approx10.96.
$$

To obtain an upper estimate for $a_0$, we consider the following test function
$$
\bar\varphi(x)=\left(\begin{array}{c}\bar f_{R}(\vb x\vb)\\0\end{array}\right)
$$
where $\bar f_{R}(\vb x\vb)=\bar f\left(\frac{\vb x\vb}{R}\right)$, 
$$
\bar f(\vb x\vb)=\left\{\begin{aligned}&\cos(\vb x\vb)&&|x|\le\frac{\pi}{2}\\&0&&|x|>\frac{\pi}{2}\end{aligned}\right.
$$
and $R\in(0,1)$ is such that $\int \vb \bar f_R\vb^2\,dx=1$. This implies
$$
R=\left(\frac{2}{\pi}\right)^{2/3}\left(\frac{3}{\pi^2-6}\right)^{1/3}\,.
$$
Next, we denote by $\bar a$ the positive constant such that $\mathcal E(\varphi)=0$. By definition,
$$
\bar a= \frac{2\int_{\mathbb{R}^3}{\frac{|\bm\sigma\cdot\nabla\bar\varphi|^2}{(1-|\bar\varphi|^2)_+}\,dx}}{\int \vb \bar \varphi\vb^4}= \frac{2\int_{\mathbb{R}^3}{\frac{|\nabla\bar f_R|^2}{1-|\bar f_R|^2}\,dx}}{\int \vb \bar f_R\vb^4}\,,
$$
and, by a straightforward calculation, we obtain 
$$
\int_{\mathbb{R}^3}{\frac{|\nabla\bar f_R|^2}{1-|\bar f_R|^2}\,dx}=\frac{\pi^{4}}{6}R=\frac{\pi^{10/3}}{3^{2/3}(2(\pi^2-6))^{1/3}},
$$
$$
\int \vb \bar f_R\vb^4=\frac{\pi^2(2\pi^2-15)}{32}R^3=\frac{3(2\pi^2-15)}{8(\pi^2-6)}.
$$
As a consequence,
$$
\bar{a}=\frac{8\pi^{10/3}\left(\frac{2}{3}(\pi^2-6)\right)^{2/3}}{3(2\pi^2-15)}\approx 48.06
$$
Since the energy functional $\mathcal E$ is decreasing in $a$, if $a>\bar a$ then $I\le \mathcal E(\bar \varphi)<0$. As a conclusion, $a_0\le \bar a +\varepsilon$ for all $\varepsilon>0$. 
\endproof

\appendix\section{}\label{sectionnew}
\subsection{}\label{A1}

We begin this section by proving  that if 
$(\varphi, \chi)$ a solution of (\ref{eqdiracnrl}) with  $\varphi\in H^1(\mathbb R^3)$ of the form (\ref{eqphireal}), then $|\varphi|^2\leq 1$ a.e. in $\mathbb R^3$. As we saw before, $\varphi$ is a solution of 
\begin{equation}\label{ELeqn}
-\bm\sigma\cdot\nabla\left(\frac{\bm\sigma\cdot\nabla\varphi}{1-|\varphi|^2}\right)+\frac{|\bm\sigma\cdot\nabla\varphi|^2}{(1-|\varphi|^2)^2}\varphi-a|\varphi|^2\varphi+b\varphi=0,
\end{equation}
or equivalently,
$$
\frac{\Delta\varphi}{|\varphi|^2-1}-\frac{|\bm\sigma\cdot\nabla\varphi|^2}{(|\varphi|^2-1)^2}\varphi-a|\varphi|^2\varphi+b\varphi=0\,,$$
or still,
$$\Delta\varphi-\frac{|\nabla\varphi|^2}{|\varphi|^2-1}\varphi-(a|\varphi|^2\varphi-b\varphi)(|\varphi|^2-1)=0\,,
$$
because for functions $\varphi$ of the form (\ref{eqphireal}), 
$$ |\sigma\cdot\nabla\varphi|^2 = |\nabla\varphi|^2\quad\mbox{and}\quad \sigma\cdot(\nabla\varphi\wedge\nabla\varphi)=0\quad\mbox{a. e.}$$
 For any $K>1$, we define the truncation function $T_K(s)$ by $T_K(s)=s$ if $1 < s < K$, and $T_K(s)=0$ otherwise. Multiplying the above equation by $\varphi \,T_K(|\varphi|^2)\in L^2(\mathbb R^3)$, we obtain
\begin{align}\label{eqestimatesol}
-\int_{\mathbb R^3}|\nabla\varphi|^2\, T_K(|\varphi|^2)&-\int_{\mathbb R^3}(\nabla\varphi\cdot\varphi)\,\nabla T_K(|\varphi|^2)-\int_{\mathbb R^3}\frac{|\nabla\varphi|^2}{|\varphi|^2-1}|\varphi|^2\, T_K(|\varphi|^2)\nonumber\\
&-\int_{\mathbb R^3}(a|\varphi|^2-b)(|\varphi|^2-1)|\varphi|^2\,T_K(|\varphi|^2)=0.
\end{align}
Moreover, for all $K>1$, 
$$
\nabla T_K(|\varphi|^2)=\left\{
\begin{aligned}&2\varphi\cdot\nabla\varphi&& \;1<|\varphi|^2<K\\
&0&&\;|\varphi|^2\le 1 \;\mbox{ or }\; |\varphi|^2\ge K
\end{aligned}
\right..
$$
Therefore, if $a-b>0$ the l.h.s of (\ref{eqestimatesol}) is negative  and this implies that either $|\varphi|^2\leq 1$ or  $|\varphi|^2\geq K$ a.e.  As a conclusion,  taking the limit $K\to +\infty$,  if $a-b>0$ then any solution $\varphi$ of \eqref{ELeqn} of the form \eqref{eqphireal} satisfies $|\varphi|^2 \le 1$ a.e. in $\mathbb R^3$, and in the equation \eqref{ELeqn} we can replace the term $(1-|\varphi|^2)$ by $(1-|\varphi|^2)_+$ without changing its solution set. The same happens for solutions of the form \eqref{eqsolrad}.

\subsection{}\label{A2}

Let us next prove that the functional ${\mathcal F}$,  defined by \eqref{eqfunctF}, is not bounded from below. Consider the function $\xi$ introduced in the proof of Lemma \ref{lemcc}. Let us denote $A:=\int_{\mathbb R^3} \vb\xi(x)\vb^2\,dx$. 

Then, let us define the radially symmetric function $$ f(r)= \left\{ \begin{array}{ll} e^{(r-\sqrt{\ln 2})^2}\,,\quad 0\le r<\sqrt{\ln 2}\,, \\ \bar\xi(r+1-\sqrt{\ln 2}),\quad r\ge \sqrt{\ln 2}\,,\end{array} \right.$$
where $\bar\xi(|x|)=\xi(x)$ for all $x$, and take $a:=\int_{\mathbb R^3}  f(|x|)^2\,dx$. Note that $\mbox{supp}(f)\subset [0, 1+\sqrt{\ln 2}]\;$ and $\max\limits_{0\le r \le 1+\sqrt{\ln 2}}{\frac{\left(f'\left(r\right)\right)^2}{1-f^2\left(r\right)}}\le C$, for some constant $C>0$. 

Next, for all integers $n>0$,  define the rescaled functions  $\xi_n(x):=n^{3/2}\xi(nx)$.   This change of variables leaves invariant the $L^2(\mathbb R^3)$ norm. Then for $n$ large, consider the function
$$g_n(x):=\max_{\mathbb R^3}\{\xi_n, f\}\,.$$
Note that the measure of the set $\{x\in \mathbb R^3\;;\; g_n=\xi_n\}$ tends to $0$ as $n$ goes to $+\infty$.
This function satisfies $\int_{\mathbb R^3} \vb g_n(x)\vb^2\,dx=A+a+o(1)$, as $n$ goes to $+\infty$. In order to normalize it in the $L^2$ norm, let us finally define the rescaled function $g_n^R(x):=g_n\left(\frac{x}{R}\right)$, $R>0$ and choose $R_n$ such that $\int_{\mathbb R^3} \vb g^{R_n}_n(x)\vb^2\,dx=1$. As $n$ goes to $+\infty$, $R_n\to \bar R:=(A+a)^{-1/3}>0$. We compute now the energy ${\mathcal F}$ of the vector function $\varphi^{R_n}_n$ defined by
$$\varphi_n^{R_n}(x)=\left(\begin{array}{c}g^{R_n}_n(x)\\0\end{array}\right)\,.$$
We find
$$
\mathcal F(\varphi^{R_n}_n)=\int_{\xi_n^{R_n}\ge f^{R_n}}\frac{\left(\left(\xi_n^{R_n}\right)'(r)\right)^2}{1-(\xi_n^{R_n}(r))^2}\,dx-\frac{a\,n^3\,R_n^3}{2}\int_{\xi_n^{R_n}\ge f^{R_n}}{|\xi|^4\,dx}$$
$$
\qquad\qquad\qquad+R_n\,\int_{\xi_n^{R_n}\le f^{R_n}}\frac{\left(f'\left(r\right)\right)^2}{1-f^2\left(r\right)}\,dx-\frac{a\,R_n^3}{2}\,\int_{\xi_n^{R_n}\le f^{R_n}}f(x)^4\,dx $$
$$
\quad\le -\frac{a\,n^3\,R_n^3}{2}\int_{\mathbb R^3}{|\xi|^4\,dx}+R_n\,\int_{\mathbb R^3}\frac{\left(f'\left(r\right)\right)^2}{1-f^2\left(r\right)}\,dx-\frac{a\,R_n^3}{2}\,\int_{\mathbb R^3}f(x)^4\,dx+o(n^3)\,,
$$
because whenever $\xi_n^{R_n}\ge f^{R_n}$, $\left(\xi_n^{R_n}\right)^2>1$
and because the sequence $\{R_n\}_n$ is bounded. This clearly shows that $\mathcal F$ is unbounded from below.

\subsection*{Acknowledgment}
This work was partially supported by the Grant ANR-10-BLAN 0101 (NONAP) of the French Ministry of Research. The authors would like to thank \'Eric S\'er\'e for useful comments and for the proof of Lemma \ref{lempropX}. They also thank the Isaac Newton Institute, where this paper was finalized.

\bibliographystyle{article_simona}
\bibliography{Nonrelativistic_limit}

\end{document}